\documentclass[12pt, a4paper]{article}
\usepackage{amsthm, amsfonts, amsmath, amssymb, float, graphicx, graphics, epstopdf, subfigure, epigraph, appendix, mathrsfs, enumitem}
\usepackage[left=3 cm,top=3cm,right=3 cm, bottom = 4cm]{geometry}
\usepackage[all]{xy}
\usepackage{hyperref}

\usepackage{color}

\newcommand{\naturals}{\mathbb{N}}

\newcommand{\reals}{\mathbb{R}}

\newcommand{\indicator}{\mathbf{1}}

\newcommand{\pr}{\mathbb{P}}

\newcommand{\ex}{\mathbb{E}}
\newcommand{\borel}{\mathcal{B}}

\newcommand{\Ea}{[\mathbf{E1}]}
\newcommand{\Eb}{[\mathbf{E2}]}
\newcommand{\Ec}{[\mathbf{E3}]}
\newcommand{\Ca}{[\mathbf{C1}]}
\newcommand{\Cb}{[\mathbf{C2}]}
\newcommand{\Cc}{[\mathbf{C3}]}

\newcommand{\vect}[1]{\boldsymbol{#1}}

\allowdisplaybreaks[1]

\newcounter{lemmacount} 
\newcounter{corcount} 
\newcounter{propcount} 
\newcounter{examples} 
\newcounter{remarks} 

\newtheorem{theorem}{Theorem}
\newtheorem{lemma}[lemmacount]{Lemma}

\newenvironment{remark}[1][]{\refstepcounter{remarks}\medskip\noindent\textbf{Remark \theremarks.$\;$}}{\medskip}

\DeclareMathOperator*{\argmax}{arg\,max}

\begin{document}
{\small

\title{On explicit form of the stationary distributions for a class of bounded Markov chains}

\author{S.\ McKinlay\footnote{
Department of Mathematics and Statistics, University of Melbourne, Parkville 3010, Australia. E-mail: s.mckinlay@ms.unimelb.edu.au.}
\ and \ K.\ Borovkov\footnote{
Department of Mathematics and Statistics, University of Melbourne, Parkville 3010, Australia. E-mail: borovkov@unimelb.edu.au.}
}

\date{}

\maketitle

\begin{abstract}
\noindent
We consider a  class of discrete time   Markov chains with state space $[0,1]$ and the following dynamics. At each time step, first the direction of the next transition is chosen at random with probability depending on the current location. Then the length of the jump is chosen independently as a random proportion of the distance to the respective end point of the unit interval, the distributions of the proportions being fixed for each of the two directions.  Chains of that kind were subjects of a number of studies and are of interest for some applications. Under simple broad conditions, we establish the ergodicity of such Markov chains and then derive closed form expressions for the stationary densities of the chains when the proportions are beta distributed with the first parameter equal to $1$. Examples demonstrating the range of stationary distributions for processes described  by this model are given, and an application to a robot coverage algorithm is discussed.
\vspace{0.2cm}

\vspace{0.2cm}

\end{abstract}


\section{Introduction}\label{S_Intro}

This paper is mostly devoted to deriving explicit formulae for the stationary densities for a class of ergodic $[0,1]$-valued discrete time Markov chains  that appear in some interesting  applications (see e.g.\ Section $4$ in \cite{Iacus}, Section 5 in \cite{Ramli}, and Section \ref{S_ex} below). The chain dynamics are as follows. Let $F_L$ and $F_R$ be two fixed distributions on $[0,1],  $ $p: [0,1] \rightarrow [0,1]$  a fixed measurable function. Given the chain value $X_n = x\in [0,1]$ at time $n$, the next jump of the process is to the left with probability $p(x)$ or to the right with probability $1-p(x)$.  If the jump is to the left, its length is given by an independent random proportion $L_{n+1} \sim F_L$ of the length of the interval $[0,x]$. Otherwise, the chain jumps to the right for a distance given by an independent random proportion $R_{n+1} \sim F_R$ of the length of the interval $[x,1]$.

That is,   the evolution of our  Markov chain $X = \{X_n\}_{n \ge 0}$ is given by
\begin{equation}\label{RDE}
X_{n+1} = X_n - X_n L_{n+1}I_{n+1} + (1-X_n) R_{n+1} (1-I_{n+1}), \quad n = 0, 1, \ldots,
\end{equation}
where $X_0 = x_0 \in [0,1],$ $I_{n+1} := \indicator_{\{U_{n+1} < p(X_n)\}}$, $\indicator_A$ being the indicator of the event $A$, and $\{L_n\}_{n \ge 1}$, $\{R_n\}_{n \ge 1}$ and $\{U_n\}_{n \ge 1}$ are independent sequences of i.i.d.\ random variables such that  $L_n\sim F_L,$ $R_n\sim F_R,$ and $U_n \sim U[0,1]$, the uniform law on $[0,1].$

The above model was apparently first introduced in Section $2.1$ of \cite{Diaconis} for $p(x) \equiv p$ and $F_L=F_R=U[0,1]$. The case of non-constant $p(x)$ was also discussed, but not pursued in \cite{Diaconis}. Further special cases of that model for various choices of $p(x)$ and   distributions $F_L$, $F_R$ were considered in \cite{Stoyanov}, \cite{Bialkowski}, \cite{Stoyanov2}, \cite{Pacheco} and \cite{Ramli}. In particular, \cite{Pacheco} dealt with the case when $p(x) \equiv x$ and $F_L=F_R=\beta(1,z)$, $z>0$, where $\beta(a,b)$ denotes the beta distribution with density
\begin{equation}\label{betapdf}
 x^{a-1} (1-x)^{b-1}/B(a,b), \quad 0<x<1,
\end{equation}
$B(a,b)$ being the beta function, $a,b>0$. It was shown in \cite{Pacheco} that in that case $X$ was ergodic with stationary law $\beta(z,z)$. The case when $F_L=F_R=U[0,1]$  and $p(x)$ is piecewise continuous was studied in \cite{Ramli}.

In the present  paper, in the case where $F_L= \beta(1,l)$, $F_R=  \beta(1,r)$ for some $l,r>0$, we derive the stationary density for piecewise continuous $p(x)$ satisfying a natural  condition that ensures ergodicity. In particular, for $F_L=F_R=\beta(1,z)$, $z>0$, and linear $p(x) = cx + (1-b)(1-x)$, $b,c \in (0,1]$, the Markov chain $X$ is ergodic with stationary distribution $\beta(bz,cz)$ (see Section \ref{S_linear} of this paper). We find that many of the existing results on the form of the stationary density for the above model are special cases of our more general Theorem~\ref{t_main} below. We also show how the same approach can be used to compute (at least, numerically) the stationary distribution when $ F_L= \beta(l_1,l_2)$, $F_R= \beta(r_1,r_2)$, $l_1,r_1 \in \naturals$, $ l_2,r_2 >0,$ by solving a two-point boundary value problem for a system of $l_1+r_1$ ordinary differential equations.

One of the main reasons for  considering such more general models is that the class of their stationary laws is far richer than in the special case $l=r=1$. In particular, by choosing large enough  $l$ and $r$, one can obtain  unimodal and multimodal stationary densities with arbitrarily high and ``sharp'' peaks. We will use this feature to generalise the robot coverage algorithm from \cite{Ramli} (see Section \ref{S_robot} below).

The Markov chain $X$ is also a special case of the so called ``give-and-take" model that was  introduced, in its deterministic form, in the context of human genetics in \cite{Li} (see also \cite{McKinlay} for an extension of this model to higher dimensions) and then studied in \cite{DeGroot}. In that model, two players (call them players~1 and~2, resp.),   with a fixed total amount of capital (normalised to be one  for convenience), at each step exchange random amounts of their fortunes with each other. The Markov chain $X$ is a version of that model where at each step only one player exchanges her fortune with the other. At time $n \ge 0$, if $I_{n+1} =0$, player~1 (whose fortune is given by~$X_n$) receives a proportion $R_{n+1}$ of player~2's fortune (which is given by $1-X_n$), otherwise if $I_{n+1} =1$, player~2 receives a proportion $L_{n+1}$ of player~1's fortune.

In Section \ref{S_main}, we establish the ergodicity of the Markov chain $X$ under simple conditions  and derive  the form of its stationary density in the ergodic case when $F_L= \beta(1,l)$, $F_R=  \beta(1,r),$ $l,r>0$. Some examples (extending some results from  \cite{Bialkowski}, \cite{Stoyanov2}, \cite{Pacheco}, and~\cite{Ramli}) are presented in Section~\ref{S_ex}.


\section{Main Results}\label{S_main}

First we  show that the following three conditions imply the ergodicity of the Markov chain~$X$. They are by no means necessary for ergodicity, but are quite suitable for the purposes of this paper.

\begin{itemize}
\item[{$\Ea$}] For some $\delta \in (0,1/2)$,
\[
\sup_{x \in [0,\delta]} \max \{p(x), 1- p(1-x)\} =:1-\varepsilon <1.
\]

\item[{$\Eb$}]  For some $\delta \in (0,1/2)$,
\[
\max\bigl\{  F_L(1-\delta),F_R(1-\delta) \bigr\}=:1-\eta <1.
\]

\item[{$\Ec$}] There exist $\delta \in (0,1/2)$ and  $s,t\in (\delta, 1-\delta)$, $ s < t $, such that   $F_L$ and $F_R$ have densities $f_L$ and $f_R$ on the intervals $(1-t-\delta,1-s)$ and $(s-\delta,t)$, respectively. Moreover, for
\begin{equation*}
g(z) := \min \left\{ \inf_{y \in [1-\delta,1]}  f_L \biggr( \frac{y-z}{y} \biggr), \inf_{y \in [0,\delta]}    f_R \biggr( \frac{z-y}{1-y} \biggr) \right\},\quad z\in (s,t),
\end{equation*}
one has
\begin{equation*}
\gamma := \int_s^t g(z) \, dz >0.
\end{equation*}
\end{itemize}

\begin{lemma}\label{ergodic}
If the Markov chain $X$ given by \eqref{RDE} satisfies conditions $\Ea$--$\Ec$ with a common $\delta,$ then $X$ is ergodic.
\end{lemma}

\begin{proof}
Let $\tau_B (x) := \inf \{k \ge 1: X_k(x) \in B\}$, where   $X_n(x)$ denotes the value of the chain after $n$ steps when  $X_0 = x$, and $B \in \borel([0,1])$ (the Borel subsets of $[0,1]$). It follows from Theorems $1.3$, $2.1$ in \cite{Borovkov} that $X$ will be ergodic provided that there exists a $V \in \borel([0,1])$, a probability measure $\varphi$ on $([0,1], \borel([0,1]))$, and a $q >0$  such that
\begin{equation*}
\begin{array}{ll}
\text{(I$'$)} & \pr(\tau_V(x) < \infty) = 1, \; \forall x \in [0,1],\vspace{1mm}\\
\text{(I)} & \sup_{x \in V} \ex \tau_V(x) < \infty,\vspace{1mm}\\
\text{(II)} & \pr(X_{1}(x) \in B) \ge  q \varphi(B), \; \forall x \in V, \; \forall B \in \borel([0,1])
\end{array}\\
\end{equation*}
(note that condition (II) implies aperiodicity of $X$).

Set $V := [0,\delta] \cup [1-\delta,1]$.  As $\pr(X_1(x) \in V) \ge \eta,$ $x \in [0,1]$, by condition $\Eb$,   the standard argument yields that, for any $x_0\in [0,1] $  and $n \ge 1$,
\begin{align*}
\pr(\tau_V(x_0) > n) &= \pr \Biggr( \bigcap_{k=1}^n \{ X_k(x_0) \notin V \} \Biggr)\\
&= \underbrace{\int_{V^c} \cdots \int_{V^c}}_{n-1} \left[ \prod_{k=1}^{n-1} \pr (X_1(x_{k-1}) \in dx_k) \right]\pr (X_1(x_{n-1}) \notin V)\\
& \le (1-\eta) \underbrace{\int_{V^c} \cdots \int_{V^c}}_{n-2} \left[ \prod_{k=1}^{n-2} \pr (X_1(x_{k-1}) \in dx_k) \right] \pr (X_1(x_{n-2}) \notin V)\\
& \le (1-\eta)^n.
\end{align*}
Hence  (I$'$) and (I) hold true.

To show that (II) is also satisfied, let
\begin{equation}\label{var_meas}
\varphi(B) : = \frac{1}{\gamma} \int_{B \cap (s,t)} g(y) \, dy, \quad B \in \borel([0,1]).
\end{equation}
Using conditions $\Ea$ and $\Ec$, we have for $B \in \borel([0,1])$ and $x \in [0,\delta]$,
\begin{align*}
\pr(X_{1}(x) \in B) &\ge \pr(X_{1}(x) \in B \cap (s,t))\\
&= (1-p(x)) \int_{B \cap (s,t)}  \frac{1}{1-x} f_R \Bigr (\frac{y-x}{1-x} \Bigr) dy\\
&\ge \varepsilon \int_{B \cap (s,t)}  g(y) \, dy\\
&= \varepsilon \gamma \varphi(B).
\end{align*}
The same argument shows that the above lower bound also holds when $x \in [1-\delta,1]$. Therefore (II) is met with $q = \varepsilon \gamma$ and $\varphi$ defined by~\eqref{var_meas}. The lemma is proved.
\end{proof}

Now we turn to deriving closed form expressions for stationary distributions in the ergodic case when both $F_L$ and $F_R$ are absolutely continuous, with densities denoted by $f_L(x)$ and $f_R(x)$, $x \in (0,1)$, respectively. Since the transition probability of   $X$ now has density
\begin{equation}\label{kernel}
f(x,y)
 := \pr(X_1(x) \in dy)/dy
  = \left \{ \begin{array}{ll}
\frac{1-p(x)}{1-x} f_R \bigr(\frac{y-x}{1-x} \bigr), & 0 \le x < y < 1,
 \vspace{1mm}\\
 \frac{p(x)}{x} f_L \bigr( \frac{x-y}{x} \bigr),  & 0 < y < x \le 1,
\end{array}\right.\\
\end{equation}
the stationary distribution $\Pi$ (when it exists) will also be absolutely continuous with density $\pi(x)$ satisfying the usual integral equation
\begin{equation}\label{IE}
u(y) = \int_0^1 u(x) f(x,y) \, dx, \quad 0<y<1.
\end{equation}
The existence of the stationary density is obvious from the standard relations
\begin{align*}
\Pi(B) &= \int_0^1 \pr(X_1(x) \in B) \, \Pi (dx) = \int_0^1 \biggr[ \int_B f(x,y) dy \biggr] \Pi (dx)\\
&=  \int_B \biggr[ \int_0^1 f(x,y) \Pi (dx) \biggr] dy =  \int_B \pi (y) \, dy, \quad B  \in \borel([0,1]),
\end{align*}
where $\pi(y) := \int_0^1 f(x,y) \, \Pi (dx)$ and the second last equality follows from Fubini's theorem.

The case when $f_L(x) \equiv f_R(x) \equiv 1$, $x \in (0,1)$, was studied in \cite{Ramli}. In that case, one can differentiate integral equation \eqref{IE} to obtain a simple separable differential equation that is easily solved to give a closed form for the stationary density $f(x)$ (coinciding with our $f$ from \eqref{e_thm} below with $l=r=1$).

We extend this result to the case where $f(x,y)$ has the semidegenerate form: for some $N,M\ge 1,$
\begin{equation}\label{semikernel}
f(x,y) = \left\{ \begin{array}{ll}
\sum_{i=1}^N a_i (y)b_i(x), & \quad 0 \le x < y < 1,
  \vspace{1mm}\\
\sum_{j=1}^M c_j(y)d_j(x),  & \quad 0 < y \le x \le 1,
\end{array}\right.\\
\end{equation}
and the factors satisfy  the following conditions:

\begin{itemize}
\item[{$\Ca$}] all  $a_i (y), c_j (y)$ are continuous on $(0,1)$;
\item[{$\Cb$}] all  $b_i (x)$ are piecewise continuous on $[0,1)$;
\item[{$\Cc$}] all  $d_j (x)$ are  piecewise continuous on $(0,1]$.
\end{itemize}

The next assertion shows that the stationary distribution of the Markov chain $X$, with transition density \eqref{semikernel}, solves a two-point boundary value problem for a system of $N+M$ ordinary differential equations. Our result is an extension of Theorem $3.1$ in \cite{Golberg} that was proved under the more restrictive assumption that all $a_i (y)$, $b_i (x)$, $c_j (y)$, and $d_j (x)$ in \eqref{semikernel} are continuous on $[0,1]$.

\begin{theorem}\label{Golberg}
Suppose that $f(x,y)$ has the semidegenerate form \eqref{semikernel} satisfying $\Ca$--$\Cc$ and let \eqref{IE} have  a solution $u(y)$ that is integrable on $(0,1)$. Then the following claims are true.
\begin{itemize}

\item[\emph{(i)}] Any such solution $u(y)$ is continuous on $(0,1)$.

\item[\emph{(ii)}] If
\begin{align}
\alpha_i(y) := \int_0^y b_i(x) u(x) \, dx, \quad i=1,2, \ldots, N,\label{alphaInt}\\
\beta_j(y) := \int_y^1 d_j(x) u(x) \, dx,\quad j=1,2, \ldots, M,\label{betaInt}
\end{align}
then
\begin{equation}\label{sol}
u(y) = \sum_{i=1}^N a_i(y) \alpha_i(y) + \sum_{j=1}^M c_j(y) \beta_j(y), \quad 0 < y < 1.
\end{equation}

Let $S$ be the union of finite discontinuity sets for $b_i (x)$, $i=1,2,\ldots,N$, and $d_j (x)$, $j=1,2,\ldots,M$. Then
\begin{align}
\alpha_i'(y) &= b_i(y)u(y), \quad y \in (0,1) \backslash S,\label{alphaDE}\\
-\beta_j'(y) &= d_j(y)u(y), \quad y \in (0,1) \backslash S,\label{betaDE}
\end{align}
\begin{equation}\label{BC}
\alpha_i(0+) = 0, \quad i=1,2,\ldots,N, \quad \beta_j(1-) = 0, \quad j=1,2,\ldots,M.
\end{equation}
\item[\emph{(iii)}] Conversely, let $\alpha_i(y)$, $i=1,2,\ldots,N$, and $\beta_j(y)$, $j=1,2,\ldots,M$, be continuous solutions of \eqref{sol}--\eqref{BC} such that $u(y)$ is integrable on $(0,1)$. Then $u(y)$ given by \eqref{sol} is a solution of \eqref{IE}.
\end{itemize}
\end{theorem}

\begin{remark}
There appears to be a typo in Theorem 3.1 in \cite{Golberg}, which contains $c_j(x)$ in the expressions on the right hand sides of \eqref{betaInt} and \eqref{betaDE} instead of $d_j(x)$, as above.
\end{remark}

\begin{proof}
(i) We have
\begin{align*}
u(y) &= \int_0^1 u(x) f(x,y) \, dx\\
&= \sum_{i=1}^N a_i (y) \int_0^y b_i(x) u(x) \, dx + \sum_{j=1}^M c_j(y) \int_y^1 d_j(x) u(x)\, dx.
\end{align*}
Each $a_i(y)$ and $c_j(y)$ is continuous on $(0,1)$ by $\Ca$, and so the continuity of $u(y)$, $y \in (0,1)$, follows since each integral on the right hand side above is absolutely continuous as $u(x)$ is integrable on $(0,1)$ and, by $\Cb$ and $\Cc$, for any $\varepsilon > 0$, $b_i(x)$ and $d_j(x)$ are bounded on $[0,1-\varepsilon]$ and $[\varepsilon,1]$, respectively.

(ii) That \eqref{sol} holds for $u(y)$ is obvious from \eqref{IE} and \eqref{semikernel}. Differentiating \eqref{alphaInt} and \eqref{betaInt} at $y \in (0,1) \backslash S$ (which is possible as the integrand is continuous at such $y$), we obtain \eqref{alphaDE} and \eqref{betaDE}, respectively. The boundary conditions \eqref{BC} follow from the definitions \eqref{alphaInt}, \eqref{betaInt}.

(iii) Now suppose that $\alpha_i(y)$, $i=1,2,\ldots,N$, and $\beta_j(y)$, $j=1,2,\ldots,M$, are continuous solutions to \eqref{sol}--\eqref{BC}. Since, as above, for any $\varepsilon > 0$, $b_i(x)$ and $d_j(x)$ are bounded on $[0,1-\varepsilon]$ and $[\varepsilon,1]$, respectively, and $u(y)$ is integrable on $(0,1)$, we have from \eqref{alphaDE}--\eqref{BC} and the assumed continuity  that
the functions $\alpha_i(y)$ and $\beta_j(y)$ are given by the right hand sides of \eqref{alphaInt} and \eqref{betaInt}, respectively. Substituting these representations into \eqref{sol} shows that $u(y)$ satisfies \eqref{IE} with $f(x,y)$ given by \eqref{semikernel}. The theorem is proved.
\end{proof}

Theorem~\ref{Golberg} allows us to easily derive the form of the stationary distribution when
\begin{equation}\label{LRbeta}
F_L = \beta(1,l), \quad F_R = \beta(1,r), \quad l,r>0.
\end{equation}
Indeed, in this case $f_L(y) = ly^{l-1}$, $f_R(y) = r(1-y)^{r-1}$, $y \in (0,1)$, so that the transition density \eqref{kernel} for the chain has the semidegenerate form \eqref{semikernel} with $N=M=1$ and
\begin{equation}\label{abcd}
a_1(y) = r(1-y)^{r-1}, \quad b_1(x) = \frac{1-p(x)}{(1-x)^r}, \quad c_1(y) = ly^{l-1}, \quad d_1(x) = \frac{p(x)}{x^l}.
\end{equation}

\begin{theorem}\label{t_main}
Assume that $p(x)$ is piecewise continuous on $[0,1]$ and satisfies $\Ea$, and that \eqref{LRbeta} holds true. Then $X$  is ergodic with stationary density
\begin{equation}\label{e_thm}
\pi(x) = C x^{l} \left( \frac{r}{1-x} + \frac{l}{x} \right) \exp \left( -r \int_{1/2}^x \frac{p(t)}{1-t} \, dt - l \int_{1/2}^x \frac{p(t)}{t} \, dt \right), \quad 0<x<1,
\end{equation}
where $C>0$ is a normalising constant such that $\int_0^1 \pi(x) \, dx = 1$.
\end{theorem}

\begin{proof}
It follows from the assumptions that $X$ also satisfies $\Eb$ and $\Ec$ (all with a common~$\delta$) and so is ergodic by Lemma \ref{ergodic}. Since the transition probabilities have densities~\eqref{kernel}, $X$ has a stationary density $\pi$.

Clearly, the set $S$ of discontinuity points of the factors $b_i$ and $d_j$  (see Theorem~\ref{Golberg}) coincides here with the finite set of all values $x$ and $ 1-x$ such that $x$ is a discontinuity point of~$p(\,\cdot\,)$  and the factors~\eqref{abcd} satisfy conditions $\Ca$--$\Cc$.

Substituting \eqref{abcd} into \eqref{alphaDE}--\eqref{betaDE} with $u(y)$ given by \eqref{sol} and letting $\alpha(y) := \alpha_1(y)$, $\beta(y) := \beta_1(y)$, we obtain
\begin{align}
\alpha'(y)(1-y)^r  &= (1-p(y)) \left( r(1-y)^{r-1} \alpha(y) +  l y^{l-1} \beta(y) \right), \quad y \in  (0,1) \backslash S,\notag\\
-\beta'(y)y^l &= p(y) \left( r(1-y)^{r-1} \alpha(y) + l y^{l-1} \beta(y) \right), \quad y \in (0,1) \backslash S.\label{beta1}
\end{align}
Adding the equations yields
\begin{equation*}
\alpha'(y)(1-y)^r - r(1-y)^{r-1} \alpha(y) = \beta'(y) y^l +  l y^{l-1} \beta(y), \quad y \in (0,1) \backslash S,
\end{equation*}
which is equivalent to $( \alpha(y)(1-y)^r )' =  ( \beta(y) y^l)'$. Integrating and assuming that $\alpha$ and $\beta$ are continuous on $(0,1)$, we conclude that
\begin{equation}\label{alphabetarel}
\alpha(y)(1-y)^r = \beta(y) y^l + C_1, \quad y \in (0,1),
\end{equation}
where the boundary condition $\alpha(0+)=0$ ensures that $C_1 = 0$.

Substituting \eqref{alphabetarel} into the second relation in \eqref{beta1}, we obtain the separable differential equation
\begin{equation*}
\beta'(y) = -\beta(y) p(y) \left( \frac{r}{1-y} +\frac{l}{y} \right), \quad y \in (0,1) \backslash S,
\end{equation*}
with the general solution
\begin{equation}\label{betaform}
\beta(y) = C_2 \exp \left( -r \int_{1/2}^y \frac{p(t)}{1-t} \, dt  -l \int_{1/2}^y \frac{p(t)}{t} \, dt \right), \quad y \in (0,1),
\end{equation}
where $C_2$ is a constant. Now from \eqref{alphabetarel} we also have
\begin{equation}\label{alphaform}
\alpha(y) = C_2  \frac{y^l}{(1-y)^r}   \exp \left( -r \int_{1/2}^y \frac{p(t)}{1-t} \, dt  -l \int_{1/2}^y \frac{p(t)}{t} \, dt \right), \quad y\in (0,1),
\end{equation}
and we see that both $\alpha(y)$ and $\beta(y)$ are continuous and $u(y)$ given by \eqref{sol} is integrable on $(0,1)$ indeed. It follows from \eqref{abcd}, \eqref{betaform}, \eqref{alphaform} and Theorem~\ref{Golberg} (iii) that the right hand side of \eqref{e_thm} is a solution to integral equation \eqref{IE} and so is equal to the stationary density of $X$ when $C>0$ is chosen so that $\int_0^1 \pi(x) dx = 1$.
\end{proof}

\begin{remark}
Note that the above approach can  be used to compute (at least, numerically) the stationary density of $X$  when $p(x)$ satisfies $\Ea$, while $F_L=   \beta(l_1,l_2)$, $l_1 \in \naturals$, $l_2>0$, and  $F_R=  \beta(r_1,r_2)$, $r_1 \in \naturals$, $r_2>0$. Indeed, in this case we have
\begin{equation*}
f_L \biggr( \frac{x-y}{x} \biggr) = \frac1{B(l_1,l_2)}\biggr(\frac{x-y}{x} \biggr)^{l_1-1}\biggr(\frac{y}{x} \biggr)^{l_2-1}, \quad 0 < y < x \le 1,
\end{equation*}
and
\begin{equation*}
f_R \biggr(\frac{y-x}{1-x} \biggr) = \frac1{B(r_1,r_2)}\biggr(\frac{y-x}{1-x} \biggr)^{r_1-1}\biggr(\frac{1-y}{1-x} \biggr)^{r_2-1}, \quad 0 \le x < y < 1,
\end{equation*}
and so the transition probabilities \eqref{kernel} are semidegenerate with $N=r_1$, $M=l_1$. It remains to solve the two-point boundary value problem for the system of $l_1+r_1$ ordinary differential equations.

One could also use the same approach to compute the stationary density when $p(x)$ satisfies $\Ea$ and the distributions $F_L, F_R$ are finite mixtures of $\beta(1,z)$ with different $z$-values or, more generally, are of the form
\begin{equation*}
f_L(x) = \sum_{j=1}^M \mu_j (1-x)^{l_j-1}, \quad \sum_{j=1}^M \mu_j /l_j=1, \quad \mu_j \in \reals,\  l_j> 0,  \quad j=1,\ldots, M.
\end{equation*}
and
\begin{equation*}
f_R(x) = \sum_{i=1}^N \lambda_i (1-x)^{r_i-1}, \quad \sum_{i=1}^N \lambda_i/r_i =1, \quad \lambda_i \in \reals, \ r_i > 0,  \quad i=1,\ldots, N,
\end{equation*}
In this case, the transition density   \eqref{kernel} is also semidegenerate, and it remains to solve the two-point boundary value problem for the system of $N+M$ ordinary differential equations.
\end{remark}


\section{Examples}\label{S_ex}



\subsection{The case of polynomial $p(x)$}\label{S_linear}

Suppose that the function $p$ is polynomial: for a $k\in \naturals$, one has
\[
p(x) =\sum_{n=0}^k p_n x^n=\sum_{n=0}^k q_n (x-1)^n  ,
\]
where, of course, $p_n=p^{(n)}(0)/n!$ and $q_n=p^{(n)}(1)/n!$, $n\ge 0,$ and $p(x)\ge 0$ for $x\in [0,1].$ Assuming that $p_0 <1$ and $ q_0 >0$  to ensure that condition $\Ea$ is satisfied, and that \eqref{LRbeta} holds true, we see that the conditions of Theorem~\ref{t_main} are met, and so the Markov chain $X$ is ergodic. A straightforward computation of the integral in \eqref{e_thm}   shows that $X$ has stationary density of the form
\[
\pi(x) = Cx^{l(1-p_0)-1} (1-x)^{rq_0-1}
 (l+ (r-l)x) \exp \left(r\sum_{n=1}^k \frac{q_n}n (x-1)^n
 - l \sum_{n=1}^k \frac{p_n}n x^n \right),
\]
$0<x<1.$

In particular, if $p(x)$ is linear:
$p(x) =  cx+(1-b)(1-x), $  where $b,c \in (0,1]$, and $l=r=z $:
\begin{equation}
\label{common_z}
F_L=F_R= \beta(1,z), \quad z>0,
\end{equation}
we immediately obtain that $X$ has stationary distribution $\beta(bz, cz)$. This is a direct  extension of Theorem~3 in \cite{Pacheco} and Proposition~1 in \cite{Stoyanov2}, where the special cases $p(x) \equiv x$, and $p(x) = p \in (0,1)$, respectively, were considered.


\subsection{The case of piecewise constant $p(x)$}\label{S_piecewise}

Next we consider the case of piecewise constant $p(x)$, under the assumption~\eqref{common_z}. We will show, in particular,  that in that case one can obtain multimodal stationary densities, with modes located at the  discontinuity points of $p(x)$.

Suppose that
\begin{equation*}
p(x) = p_i\in [0,1], \quad s_{i-1} \le x < s_i, \quad i=1, \ldots, k,
\end{equation*}
where  $k \in \naturals$, $0=s_0 < s_1 < \cdots < s_k = 1$, and $p_1 <1$,   $p_k >0$. Then $p(x)$ satisfies $\Ea$, and Theorem~\ref{t_main} implies that  the Markov chain $X$ has stationary density
\begin{equation}\label{ex2_den}
\pi(x) = C_i x^{z(1-p_i)-1}(1-x)^{zp_i-1}, \quad s_{i-1} \le x < s_i, \quad i=1, \ldots, k,
\end{equation}
where $C_1, \ldots, C_k$ are positive constants. That is, the (continuous) stationary density of $X$ is ``glued" of pieces of different beta densities on disjoint intervals $(s_{i-1},s_i)$.

To find the constants $C_i$, we note that, by the continuity of $\pi (x)$,
\begin{equation}\label{C_rel}
C_{i} = C_{i -1}\biggr( \frac{s_{i-1}}{1-s_{i-1}} \biggr)^{z(p_{i}-p_{i-1})}= C_1
 \prod_{j=1}^{i-1} \biggr( \frac{s_j}{1-s_j} \biggr)^{z(p_{j+1}-p_j)}, \quad 2\le  i\le k.
\end{equation}
Using notation $B_x(a,b):=\int_0^x t^{a-1}(1-t)^{b-1}dt,$ $x \in [0,1]$, $a,b>0,$ for the incomplete beta function, we obtain from  the   relation   $\int_0^1 \pi(x) \, dx = 1$, \eqref{ex2_den} and \eqref{C_rel} that
\begin{align*}
1 
&= \sum_{i=1}^k  C_i  \bigl[ B_{s_i}(z(1-p_i),zp_i) - B_{s_{i-1}}(z(1-p_i),zp_i) \bigr]\\
&= C_1  \sum_{i=1}^k  \bigl[ B_{s_i}(z(1-p_i),zp_i) - B_{s_{i-1}}(z(1-p_i),zp_i) \bigr]\prod_{j=1}^{i-1} \biggr( \frac{s_j}{1-s_j} \biggr)^{z(p_{j+1}-p_j)},
\end{align*}
thus yielding $C_1$ as the reciprocal of the sum on the right-hand side, the values of $C_2, \ldots, C_k$ being now given by~\eqref{C_rel}. Note that the assertion of Theorem~1 in \cite{Bialkowski} is a  special case of the above general formula, corresponding to  $(k,z,p_1,p_2,s_1)= (2,1,0,1,1/2)$.

\begin{figure}[ht]
	\centering
	\subfigure[$p_1=0$, $p_2=1$, $s_1=0.5$]{\includegraphics[scale=0.52]{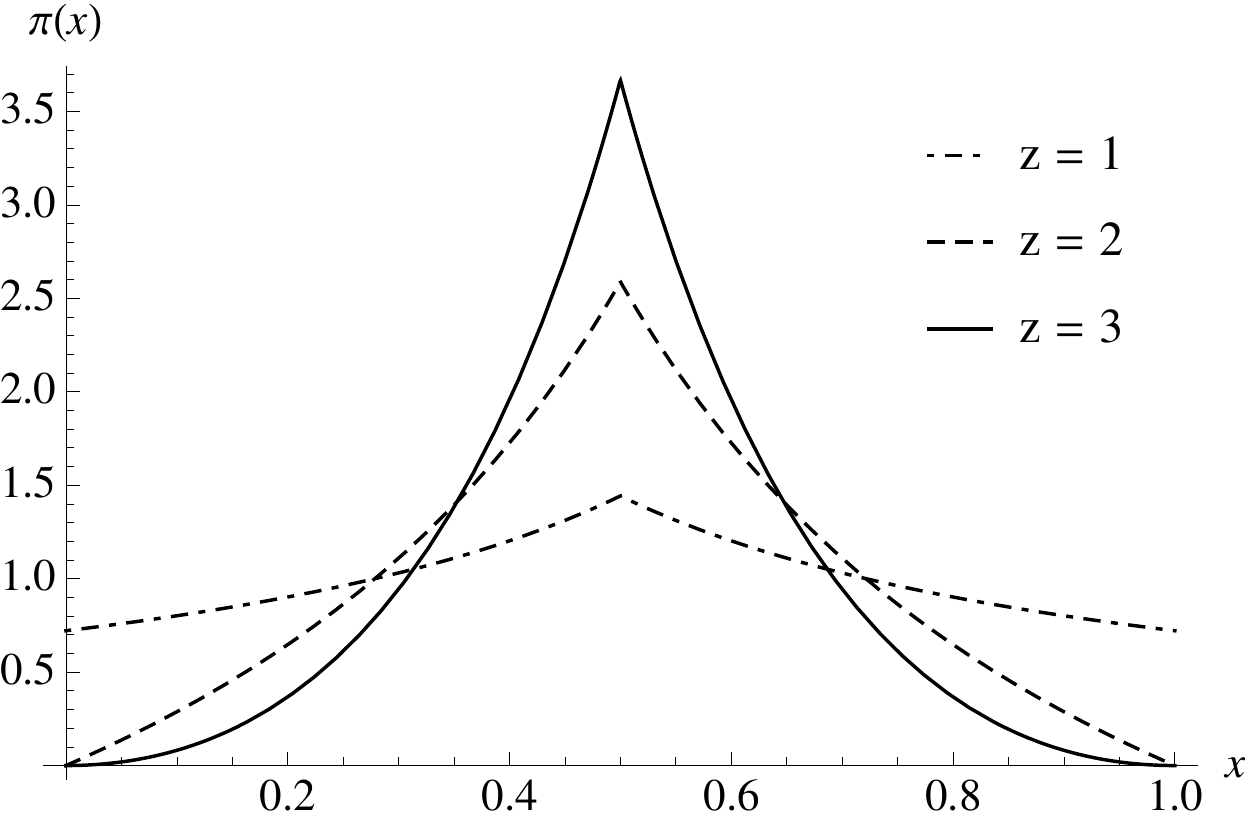}}
	  \hspace{0.3cm}
	\subfigure[$p_1=0$, $p_2=1$, $z=2$]{\includegraphics[scale=0.54]{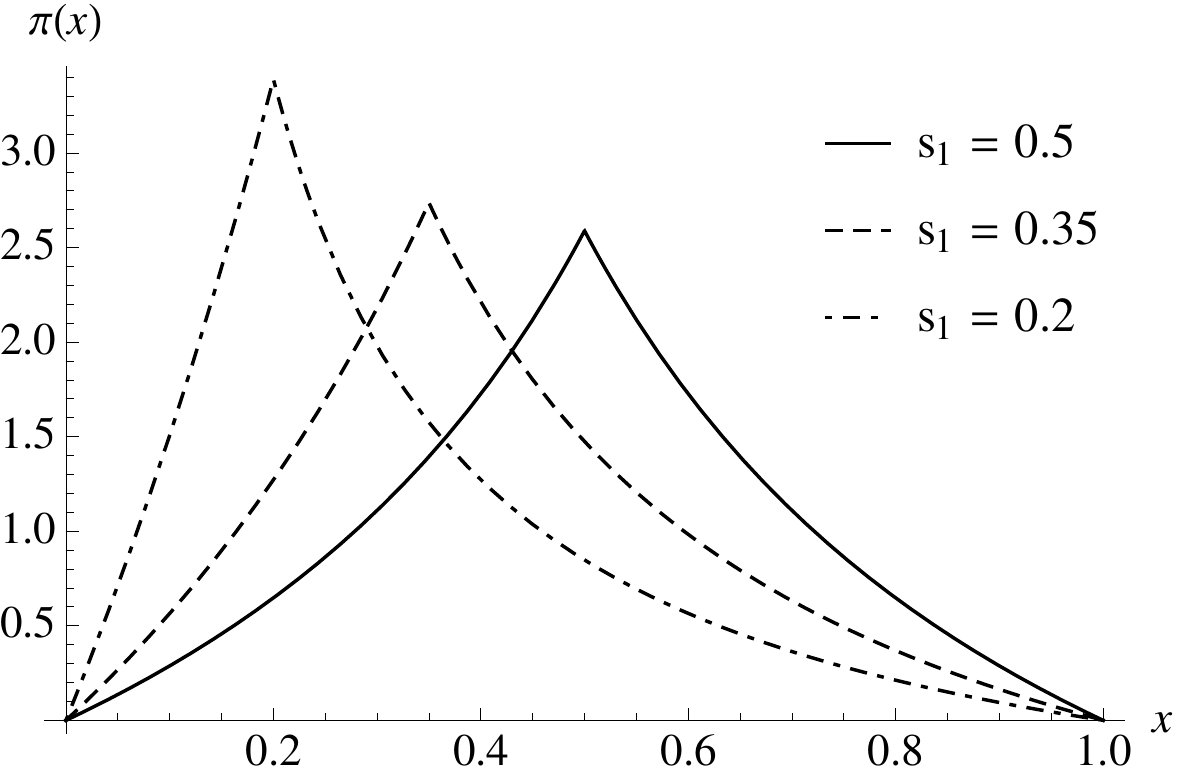}}
	\subfigure[$s_1=0.5$, $z=2$]{\includegraphics[scale=0.52]{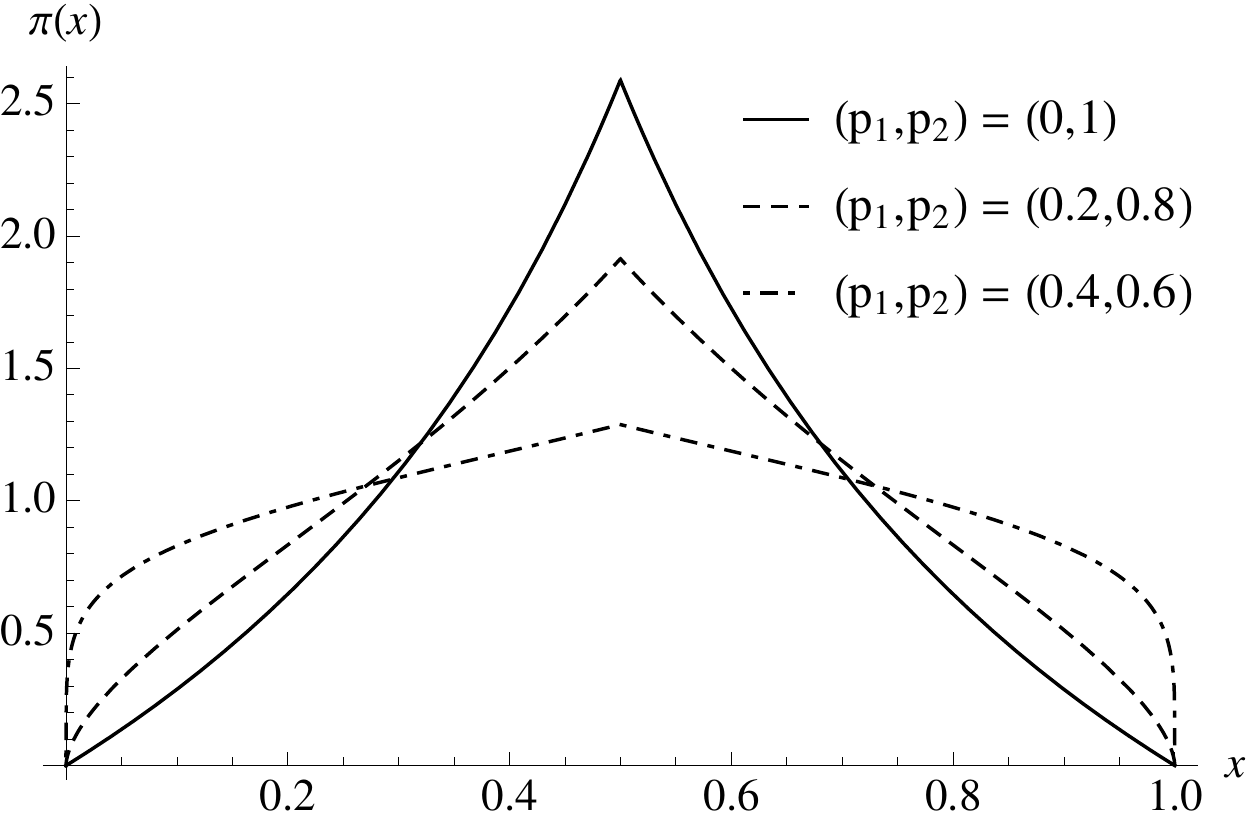}}
	  \hspace{0.3cm}
	\subfigure[$s_1=0.5$, $p_2=1$, $z=2$]{\includegraphics[scale=0.51]{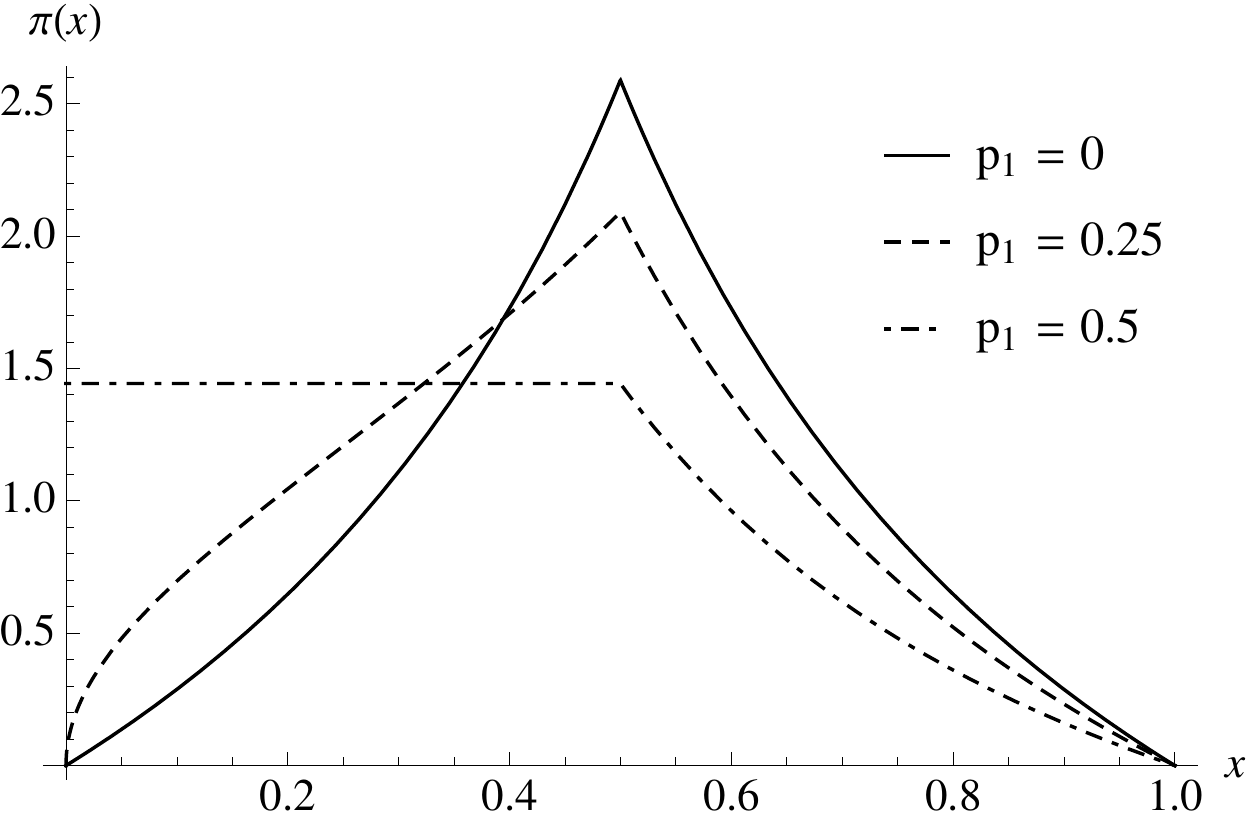}}
	\caption{\small Some examples of peaked stationary densities generated by piecewise constant $p(x)$ with $k=2$.}\label{peaked}
\end{figure}

Some examples of peaked stationary densities of the form \eqref{ex2_den}   with various choices of $s_1$ and $(k,z,p_1,p_2)= (2,1,0,1)$ were  given in~\cite{Ramli}. That, of course, is a very special case producing a rather limited range of peaked stationary densities (see Fig.~3 in~\cite{Ramli}). In the more general case where the parameter $z$ can assume any positive value, one can create a much richer variety of (arbitrarily high-) peaked stationary densities, see Fig.~\ref{peaked} for some examples of peaked stationary densities generated by piecewise constant $p(x)$ with $k=2$. Moreover, such functions $p(x)$  can also generate bimodal   stationary densities (see Fig.~\ref{bimodal}), while models with $k>2$ can have more general multimodal stationary densities (see Fig.~\ref{multimodal} for examples of multimodal densities corresponding to models with $k=6$).

\begin{figure}[ht]
	\centering
	\subfigure[$p_1=0.7$, $p_2=0.3$, $s_1=0.5$]{\includegraphics[scale=0.51]{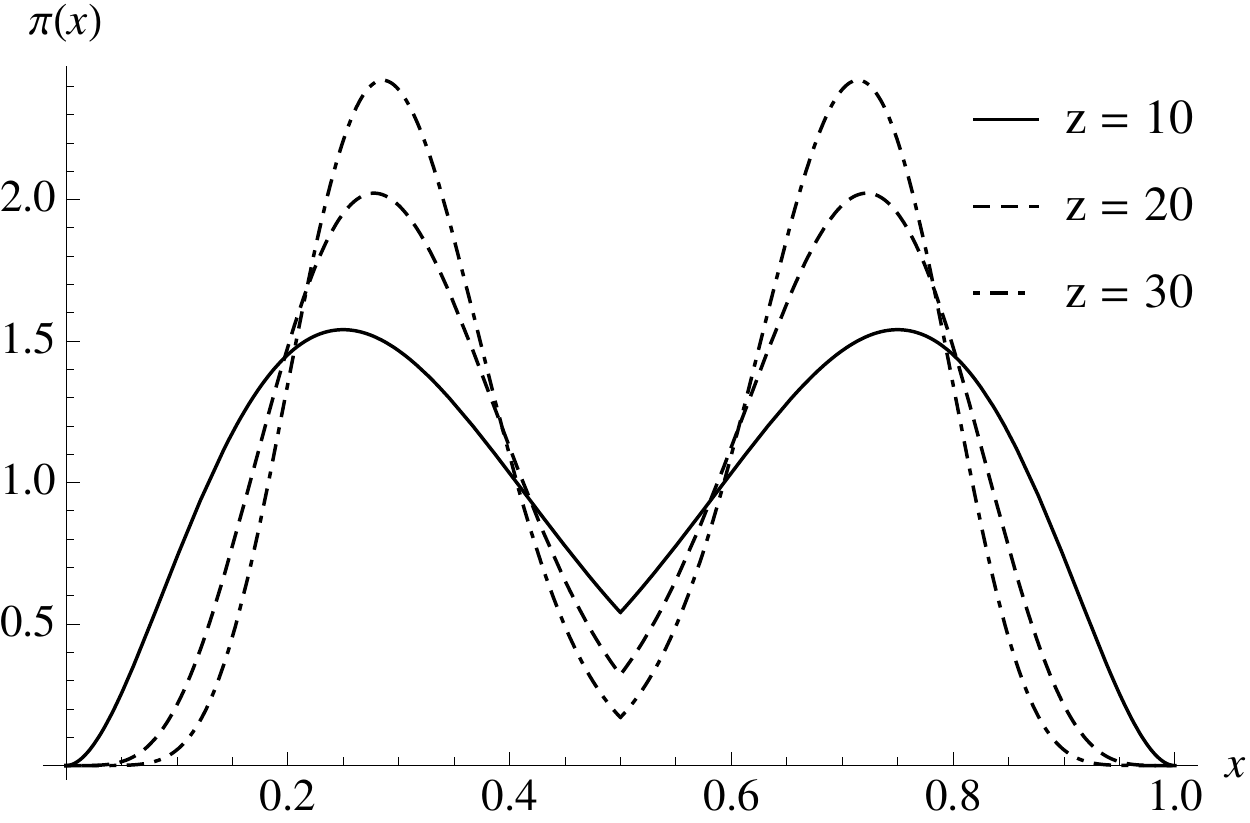}}
	  \hspace{0.3cm}
	\subfigure[$p_1 = 0.7$, $p_2 = 0.3$, $z = 10$]{\includegraphics[scale=0.51]{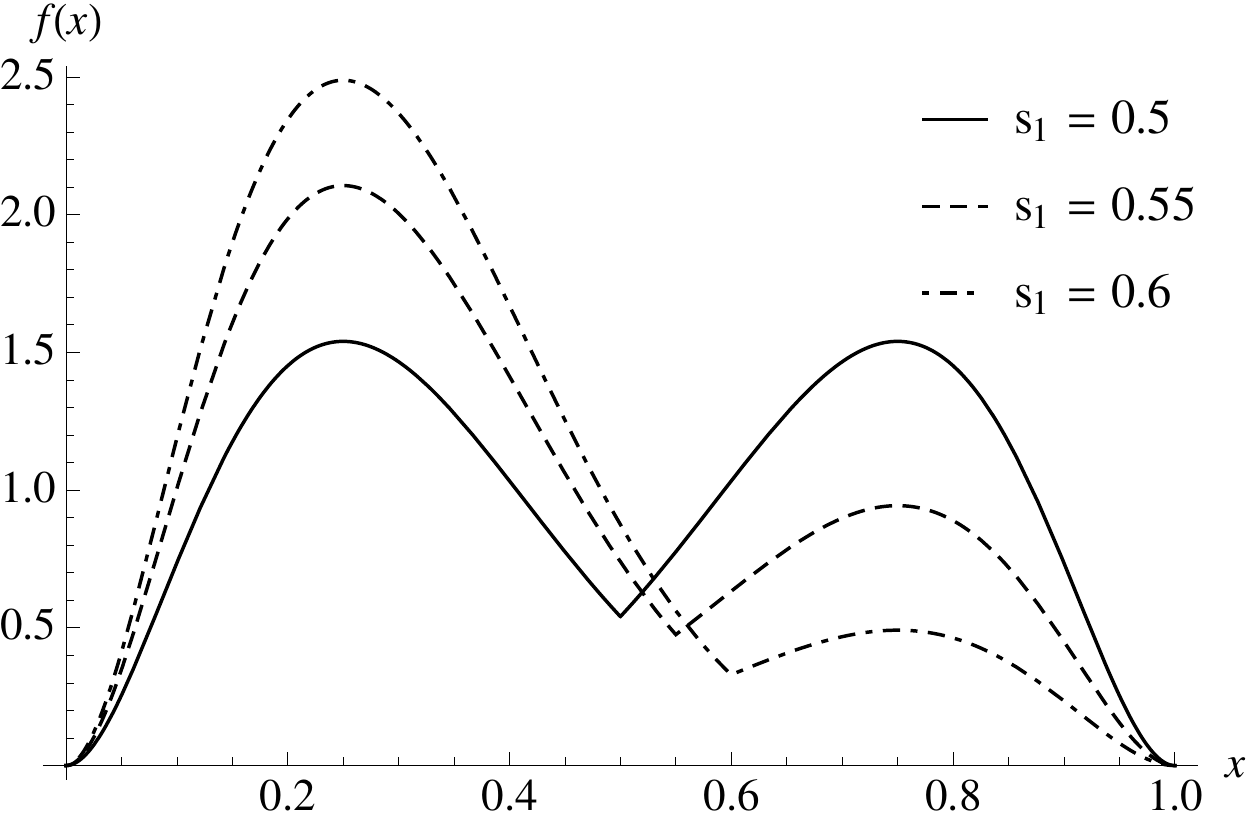}}
	\subfigure[$p_1 = 0.7$, $s_1=0.5$, $z=10$]{\includegraphics[scale=0.51]{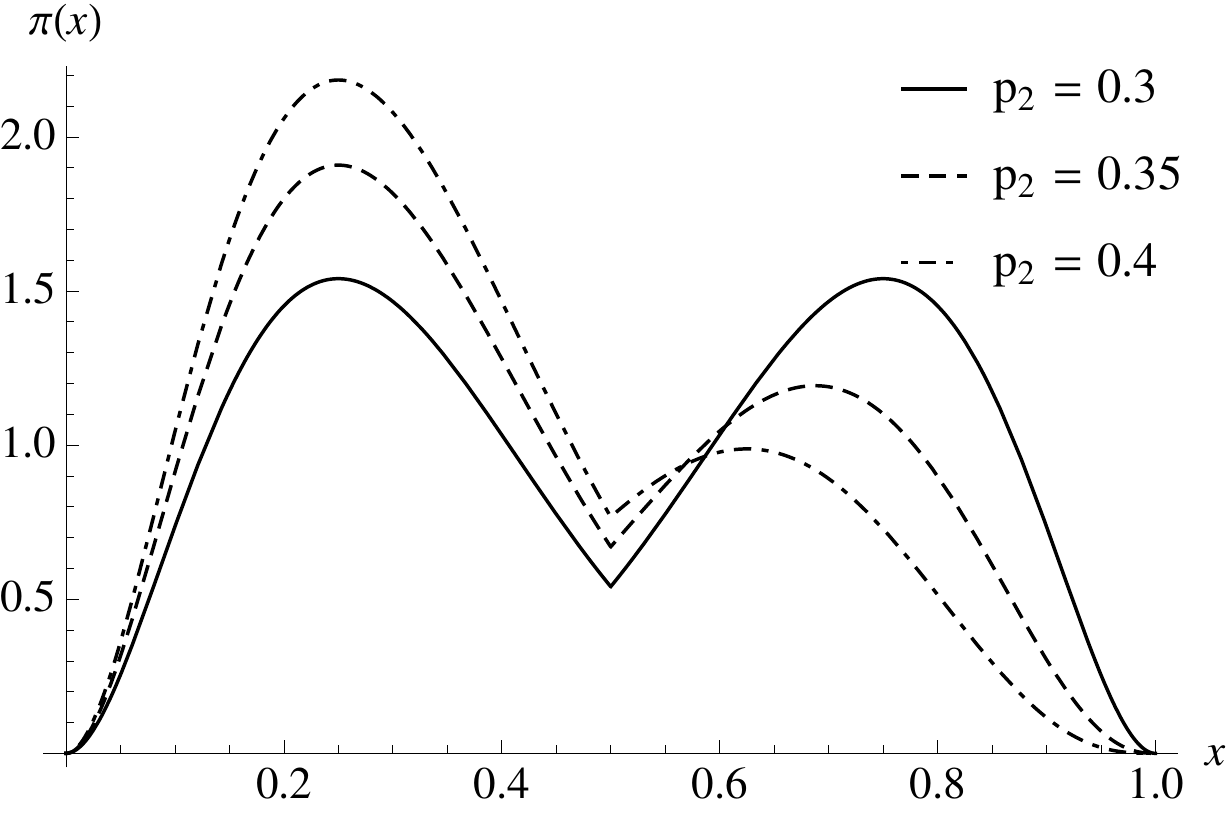}}
	  \hspace{0.3cm}
	\subfigure[$z=10$]{\includegraphics[scale=0.51]{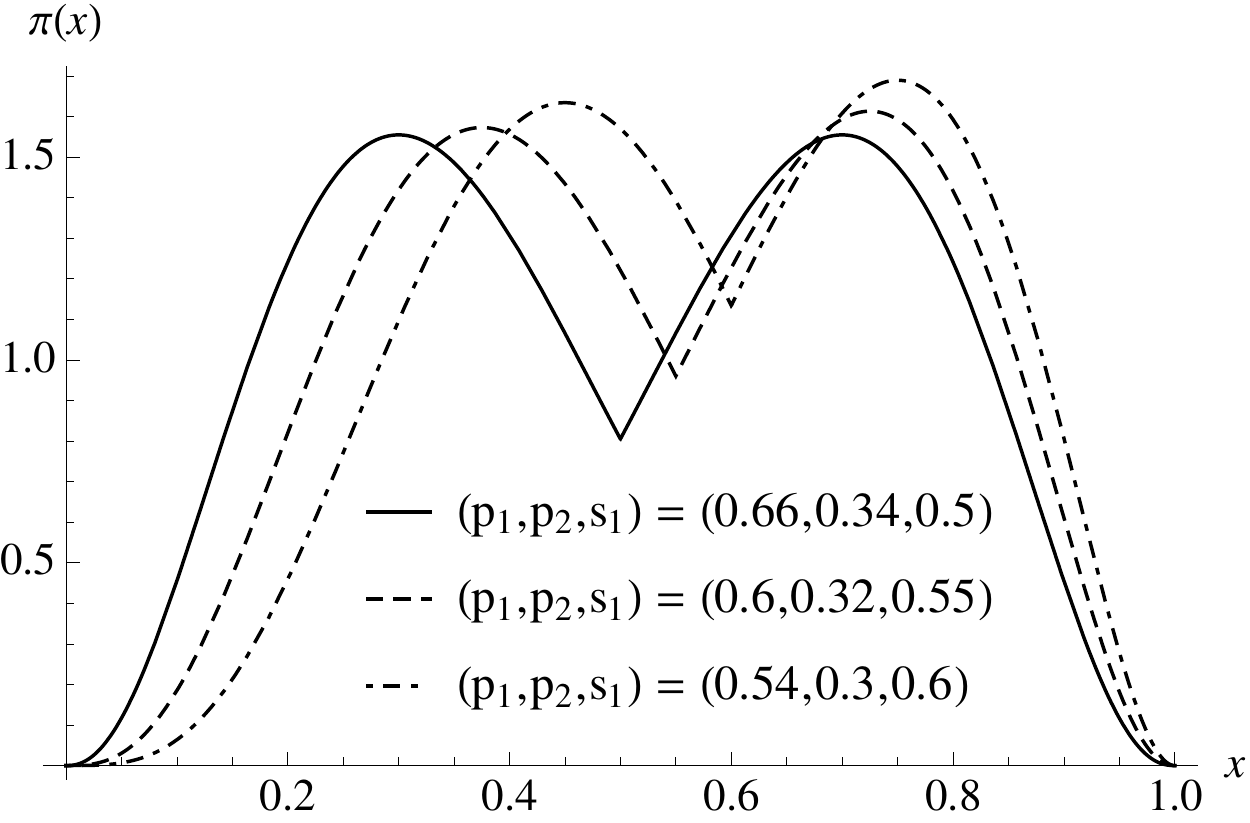}}
	\caption{\small Some examples of bimodal stationary densities generated by piecewise constant $p(x)$ with $k=2$.}\label{bimodal}
\end{figure}

\begin{figure}[ht]
	\centering
	\subfigure[$z=3$, $\vect{p} = (0,1,0,1,0,1)$]{\includegraphics[scale=0.71]{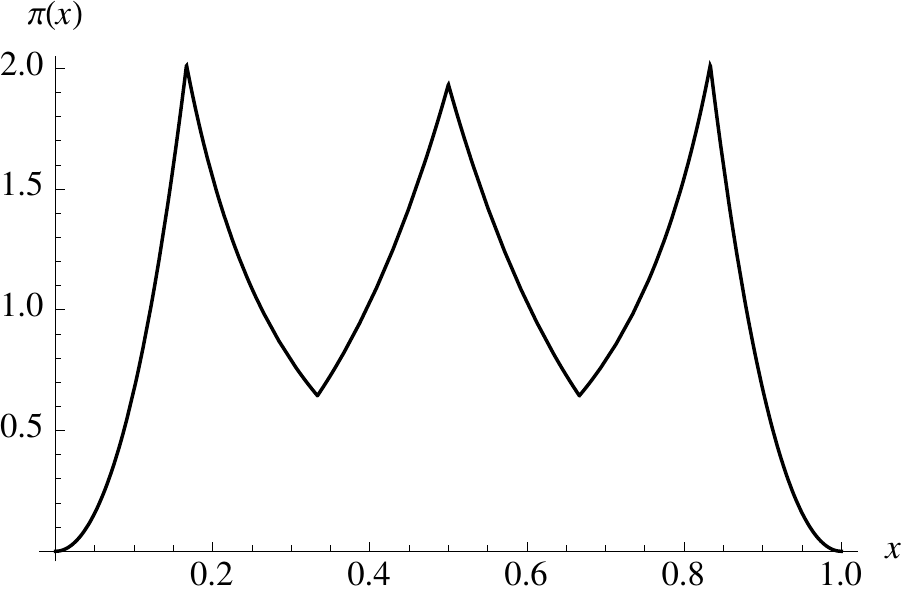}}
	  \hspace{0.3cm}
	\subfigure[$z=5$, $\vect{p} = (0,1,0,1,0,1)$]{\includegraphics[scale=0.71]{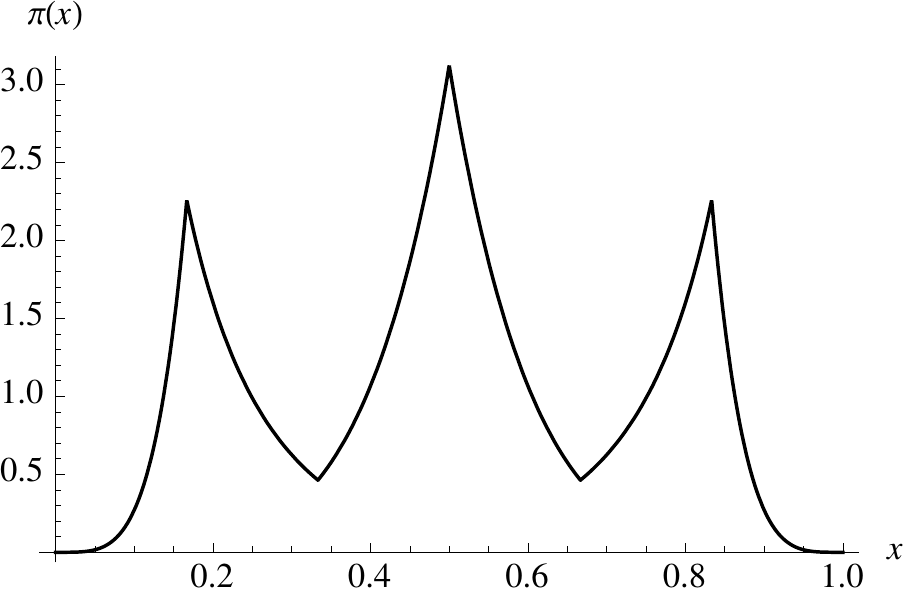}}
	\subfigure[$z=3$, $\vect{p} = (1/5,1,1/5,1,1/5,1)$]{\includegraphics[scale=0.71]{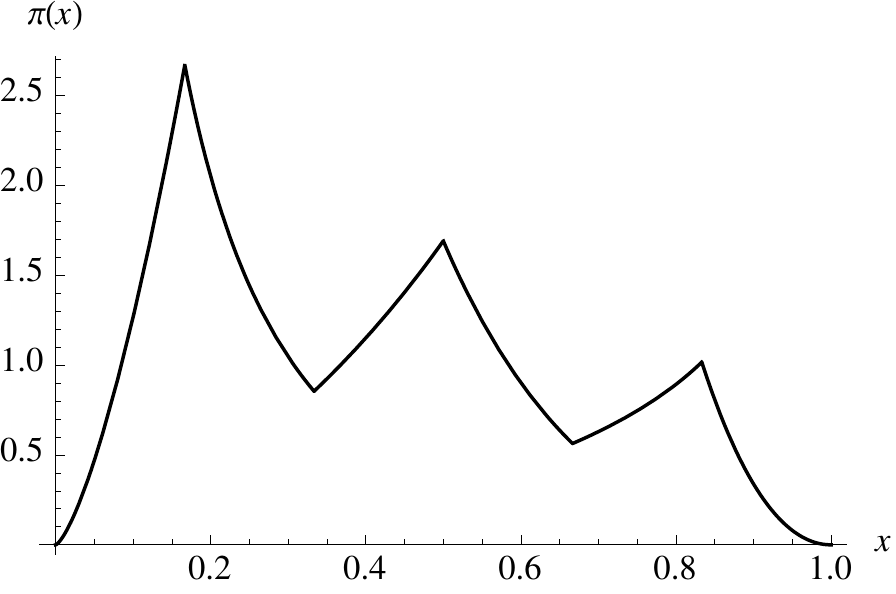}}
	  \hspace{0.3cm}
	\subfigure[$z=3$, $\vect{p} = (1/5,4/5,1/5,4/5,1/5,4/5)$]{\includegraphics[scale=0.71]{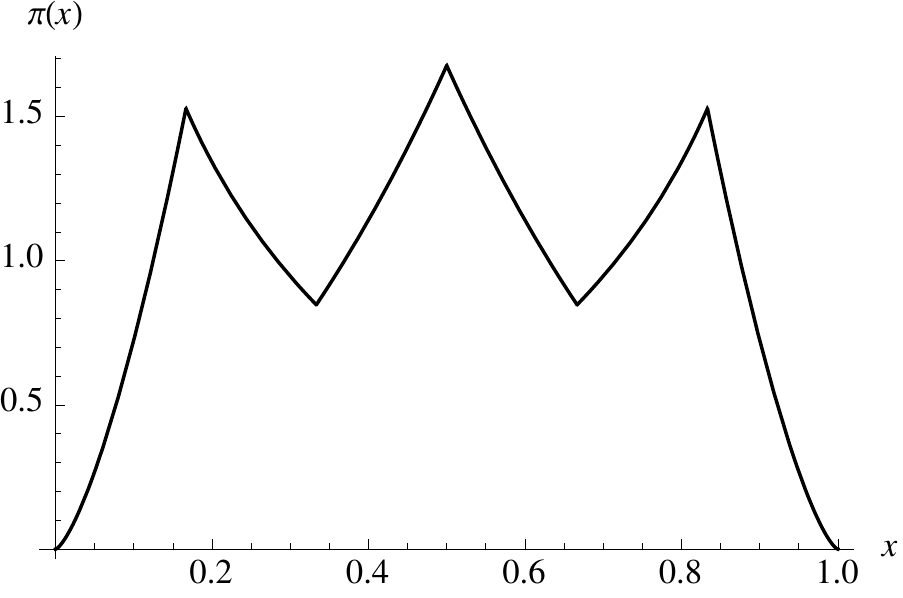}}
	\caption{\small Some examples of trimodal stationary densities generated by piecewise constant $p(x)$, where $k = 6$, $\vect{p} := (p_1, \ldots, p_6)$, $s_1=1/6$, $s_2=1/3$, $s_3=1/2$, $s_4=2/3$, $s_5=5/6$.}\label{multimodal}
\end{figure}


\subsection{A robot coverage algorithm and random search}\label{S_robot}

In this example, we discuss a generalisation of the robot coverage algorithm suggested  in Section~5 of~\cite{Ramli}, where the following scenario was considered.

Suppose a robot is moving periodically in a rectangular room of size $d_1 \times d_2$. At each location the robot stops, a measurement is taken, and then the robot moves to the next location according to some rule. The objective of the rule is to ensure the measurements cover the whole room, with certain areas in the room given higher priority. To achieve that, one can  make the robot move according to a $D:= [0,d_1] \times [0,d_2]$-valued Markov chain $\{(X_{1,n},X_{2,n})\}_{n\ge 0}$, with a given stationary density (that will specify the degree of attention the robot will be paying to different areas of the room).

Given that  the robot is at $(x_1 ,x_2) \in D$ at time $n\ge 0$, at time $n+1$ it  moves to its next  location in $ D$ according to the following algorithm.

\begin{itemize}[leftmargin=1.7cm]
\item[\textbf{Step 1:}] For two given measurable functions $p_i: [0,d_i] \rightarrow [0,1]$, $i   =1,2$, the $i$th component of the displacement vector is negative with probability $p_i(x_i)$ (and positive with probability $1-p_i(x_i)$), $i=1,2$, the signs of the two components being independent of each other.

\item[\textbf{Step 2:}] The distances $\Delta_i $  to be travelled in the $i$th dimensions, $i=1,2,$   are selected at random as follows:
\end{itemize}
\begin{align*}
\Delta_i & := \left \{ \begin{array}{ll}
R_{i,n}(d_i-x_i)  & \text{if moving in positive direction along axis $i$},\\
(1-L_{i,n})x_i   & \text{if moving in negative direction along axis $i$},
\end{array}\right.
\end{align*}
where $L_{i,n} \sim \beta(1,l_i)$, $R_{i,n} \sim \beta(1,r_i)$, $i=1,2$, $n\in \naturals$, are all independent of each other and of the choices made at Step~1.

Clearly, the stationary density on $D$ of the robot location process $\{(X_{1,n},X_{2,n})\}_{n \ge 0}$ is the product of the stationary densities of the component  processes $\{X_{1,n}\}_{n \ge 0}$ and $\{X_{2,n}\}_{n \ge 0}$. The algorithm suggested in \cite{Ramli} used indicator functions $p_i$ and uniformly distributed  $L_{i,n}$, $R_{i,n}$ only, so that the  above version allows one to design much more general ``preference functions" (i.e., stationary densities for the Markov chain describing the robot's movement) for the coverage algorithm.

Now suppose, as it was done in the example in Section~5 of~\cite{Ramli}, that there is a single point of interest at $(y_1, y_2):=(0.2 d_1, 0.5 d_2)$. Setting
\begin{equation}\label{robot_mf}
p_i (x_i) := \indicator_{\{x_i>y_i\}} , \qquad i= 1,2,
\end{equation}
we obtain a stationary density on $D$ with a single peaked mode at the point  $(y_1, y_2)$ when $l_1,r_1,l_2,r_2 \ge 1$. See Fig.~\ref{robot} below for a plot of the stationary density when $l_1=r_1=l_2=r_2=3$ and $d_1=d_2=1$.

\begin{figure}[ht]
	\centering
	\includegraphics[scale=0.20]{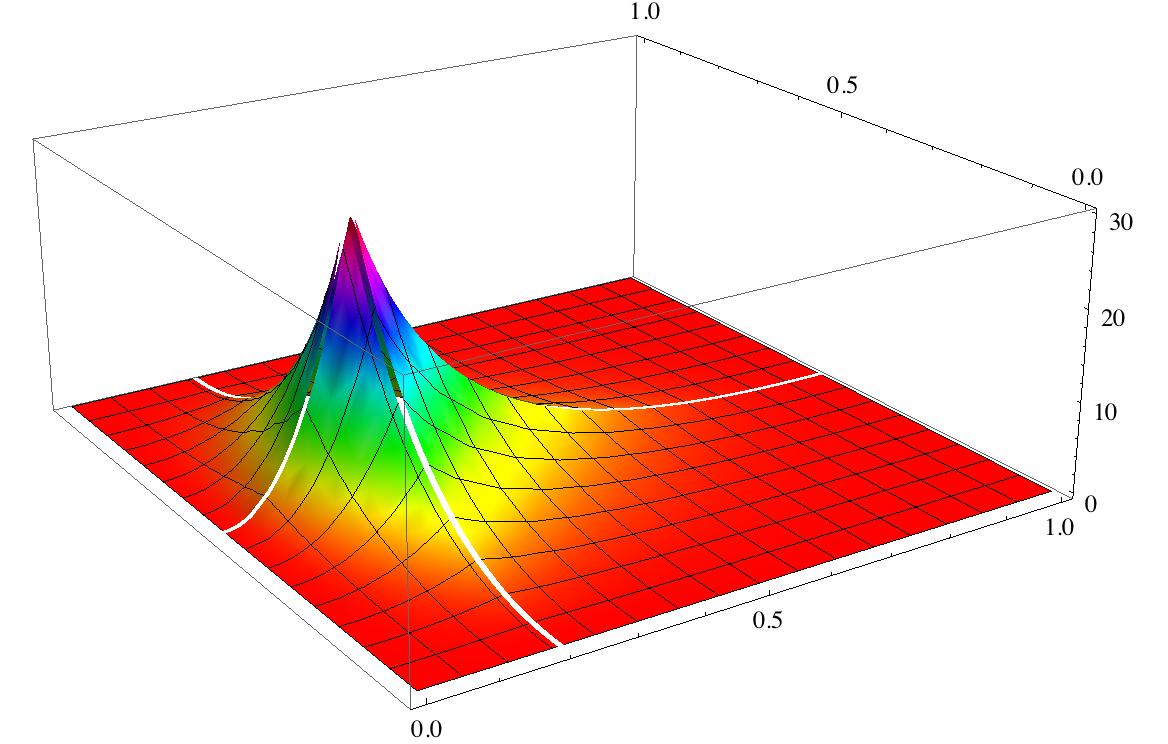}
	\caption{The stationary density of the robot converge algorithm with $p_i$ given by \eqref{robot_mf}, $l_1=r_1=l_2=r_2=3$, and $d_1=d_2=1$.}\label{robot}
\end{figure}

\begin{remark}
To design a robot coverage algorithm for a general bounded connected region  $D\subset \reals^d$, $d\ge 2$, with piecewise smooth boundary (rather than just a rectangle) and an arbitrary given ``preference function", one can use the Markov chain model suggested in \cite{Borovkov94} for simulating random vectors with given densities on such sets. That model can also be used for constructing coverage algorithms on the boundaries of the sets.

\end{remark}

In conclusion   note that the above algorithm can be modified to  adapt to the level of measurements, which will  basically turn it into  a random sequential search algorithm. More specifically, suppose that our robot measures a scalar quantity   $g(\vect{x})$ depending on the location $\vect{x}$ in the search space $D=[0,1]^d$, $d \ge 1$. The goal is to find $\vect{x^*}:=\argmax_{\vect{x} \in D} g(\vect{x})$ of the global maximum  of the objective function $g: D \rightarrow \reals$.  What distinguishes this setup from the usual optimisation problem is that one  now aims to minimise not the amount  of computation  required to find a satisfactory approximation to the maximum point but, rather, the {\em distance traveled by the robot\/} in the process.

In a typical sequential random search algorithm (see  e.g.\ Chapter 1 in \cite{Zabrinski}), given the current  ``best-found position" $\vect{y}$ (with the largest  value of $g$ among all the points of $D$ visited so far), a new candidate point  $\vect{x}$  is generated at random according to a distribution depending on the current position and the ``past search history". If $g(\vect{x}) > g(\vect{y})$ then we move to the new position~$\vect{x}$, otherwise a new candidate point is generated, according to the same distribution (or its modification). It is well understood that it is important to incorporate a ``systematic search-domain reduction into random optimisation" procedure (see e.g.\ \cite{Spaans}).  In   classical implementations of the search procedure where the new candidate point is sampled from the uniform distribution on a sphere or cube centred at $\vect{y}$  (see e.g.\ \cite{Schrack,Luus}), this is achieved, roughly speaking,  by ``shrinking" the size of the respective set (sphere or cube) at an exponential rate. Alternatively, one can try to achieve basically the same effect by changing the shape of the sampling distribution  (akin to changing the ``temperature" in simulated annealing). That can be achieved using our results on the ``peaked shape" of the stationary distribution of Markov chains.

One can construct such an algorithm as follows. Fix  a value  $v \in [0,1/2]$ that will specify our function $p(x)$, and choose  a sequence $\{z_n> 0\}_{n\ge 0}$, $z_n \uparrow \infty$ as $n\to\infty$. 
 
\begin{itemize}[leftmargin=1.7cm]
\item[\textbf{Step 0:}] Initialise algorithm parameters: an initial point $\vect{X}_0 = (X_{0,1}, \ldots, X_{0,d})$ and the iteration index   $n:=0.$ Set $\vect{Y}:=\vect{X}_0.$

\item[\textbf{Step 1:}]   Generate   $\vect{X}_{n+1}$, of which  the components $X_{n+1,j}$ are obtained from the respective values of $X_{n,j}$ according to transitions in $d$ independent  Markov chains of the form \eqref{RDE} with  $ F_L=F_R =\beta(1,z_{n})$, $
p(x) = 1-v + (2v-1) \indicator_{ \{x < Y_{j}\}}$, $j=1,\ldots, d.$

\item[\textbf{Step 2:}] If $ g (\vect{X}_{n+1})> g (\vect{Y}  ) $   then set  $\vect{Y} :=\vect{X}_{n+1} $ to update the best-found point.
    
\item[\textbf{Step 3:}] Set $n:=n+1$ and go to Step~1.

\end{itemize}
The procedure continues until a suitable stopping criterion is satisfied (e.g., the total travel distance reaches a prescribed level etc.).


\vspace{0.3cm}
\noindent \textbf{Acknowledgements.} This research was supported by the ARC Centre of Excellence for Mathematics and Statistics of Complex Systems, the Maurice Belz Trust and ARC Discovery Grant DP120102398. The first author wishes to thank the School of Mathematical Sciences at Queen Mary, University of London, for providing a visiting position while this research was undertaken.


}
\end{document}